\documentclass[10pt]{amsart}

\usepackage{amsmath, amsfonts, amsthm, amssymb, graphicx, fullpage, enumerate, float, caption, subcaption, tikz, hyperref, scalerel, boldline, makecell, longtable, array}
\usepackage{algorithm}
\usepackage{algpseudocode}
\usetikzlibrary{patterns}
\usetikzlibrary{graphs}  
\usetikzlibrary{calc}
\newtheorem{theorem}{Theorem}[section]        
\newtheorem{lemma}[theorem]{Lemma}
\newtheorem{corollary}[theorem]{Corollary}

\theoremstyle{remark}      
\newtheorem*{rem}{Remark} 
     
\theoremstyle{definition}

\def\N{\mathbb{N}}
     
\def\Q{\mathbb{Q}}          
\def\R{\mathbb{R}}     
\def\Z{\mathbb{Z}}

\def\sq{\square}

\begin{document}
\title{The sum-product problem for small sets II}  
\author{Phillip Antis, \quad Holden Britt, \quad Caleigh Chapman, \\ Elizabeth Hawkins, \quad Alex Rice, \quad Elyse Warren}
 
\begin{abstract} We establish that every set of $k=10$ natural numbers determines at least $30$ distinct pairwise sums or at least $30$ distinct pairwise products, as well as the analogous result for $k=11$ and at least $34$ sums/products, with sharpness uniquely (up to scaling) exhibited  by $\{1, 2, 3, 4, 6, 8, 9, 12, 16, 18\}$ and $\{1, 2, 3, 4, 6, 8, 9, 12, 16, 18, 24\}$, respectively. This extends previous work of the fifth author with Clevenger, Havard, Heard, Lott, and Wilson \cite{SP2023}, which established the corresponding thresholds for $k\leq 9$. Included are several classification results, including for sets of $10$ real numbers (resp. positive real numbers) determining at most $29$ pairwise sums (resp. pairwise products) that do not contain $8$ elements of any single arithmetic progression (resp. geometric progression), as well as some observations controlling additive quadruples in small subsets of two-dimensional generalized geometric progressions.  
 
\end{abstract}
 
\address{Department of Mathematics, Millsaps College, Jackson, MS 39210}   
\email{antispj@millsaps.edu}
\email{britthm@millsaps.edu}
\email{chapmce@millsaps.edu} 
\email{hawkie@millsaps.edu} 
\email{riceaj@millsaps.edu}
\email{warreek@millsaps.edu}

\maketitle   
\setlength{\parskip}{5pt}   
 
\section{Introduction}

Beginning with work of Erd\H{o}s and Szemer\'edi \cite{ES} in the early 1980s, a central question in additive combinatorics, broadly known as the \textit{sum-product problem}, concerns the extent to which finite subsets of the integers (or other rings) can simultaneously determine relatively few distinct pairwise sums and relatively few distinct pairwise products. As usual, we define $A+A=\{a+b:a,b\in A\}$ and $AA=\{ab:a,b\in A\}$, which make sense in any ring but we apply with $A\subseteq \R$. Using $|X|$ to denote the number of elements of a finite set $X$, it is a standard exercise to verify that if $A\subseteq\R$ and $|A|=k$, then $2k-1\leq |A+A| \leq (k^2+k)/2$, with equality on the left if and only if $A$ is an \textit{arithmetic progression}, a set of the form $\{x,x+d,\dots,x+(k-1)d\}$ with $d\neq 0$, and equality on the right if and only if every sum $a+b$ with $a,b\in A$, $a\geq b$ is distinct, in which case we say $A$ is a \textit{Sidon set}. Further, these inequalities transfer to product sets for $A\subseteq (0,\infty)$, as we can let $L=\{\log(a): a\in A\}$, so $AA = \{e^{a+b}: a,b\in L\}$, hence $|AA|=|L+L|$, and $L$ is an arithmetic progression if and only if $A$ is a \textit{geometric progression}, a set of the form $\{x,rx,\dots,r^{k-1}x\}$ with $r\neq 1$.

We encapsulate the sum-product problem in the positive integers by defining $$SP(k)=\min_{\substack{A\subseteq \N \\ |A|=k}}\left(\max\{|A+A|,|AA|\}\right)$$ for $k\in \N$. In other words, $SP(k)$ is the greatest integer $n$ such that every set of $k$ positive integers must determine at least $n$ distinct sums or at least $n$ distinct products. The majority of sum-product literature concerns the asymptotic growth rate of $SP(k)$, or related quantities, the most coveted conjecture \cite{ES} being $SP(k) = k^{2-o(1)}$, where $o(1)$ indicates a quantity tending to $0$ as $k\to \infty$. In other words, it is believed that at least one of $A+A$ and $AA$ must be roughly as big as possible.  Following decades of incremental progress, mostly through connections with incidence geometry (see, in chronological order, \cite{Nath}, \cite{Elekes}, \cite{Ford}, \cite{Soly}, \cite{KS2}, \cite{KS1}, \cite{Shak}, \cite{RudSS}, \cite{Bloom}), the current best-known lower bound is $SP(k)\geq k^{\frac{4}{3}+\frac{10}{4407}-o(1)}$, due to Cushman \cite{Cushman}.

Recently, more elementary methods have been applied to sum-product type questions for small sets. O'Bryant \cite{KOB} constructed an extensive dataset of pairs $(|A+A|,|AA|)$ for $A\subseteq \N$ with $|A|=k$, which is at least $84\%$ complete for all $k\leq 32$, but only known to be complete for $k\leq 6$. Previous work of the fifth author with Clevenger, Havard, Heard, Lott, and Wilson \cite{SP2023} more specifically pursued exact values of $SP(k)$ for small values of $k$, with results summarized below.

\begin{theorem}\cite[Theorem 1.1]{SP2023}\label{oldSP} We have the following exact values for $SP(k)$: $$SP(k)=\begin{cases} 3k-3, & 2\leq k\leq 7 \\ 3k-2, & k=8,9 \end{cases}. $$
\end{theorem}

Crucial ingredients to the proof of Theorem \ref{oldSP} are classification results of Freiman, originally proven in the integers but valid in any torsion-free abelian group, which describe the structure of sets with close to as few distinct sums as possible. 

\begin{theorem}\cite[p. 11]{Frei73} \label{Fclass1} If $A\subseteq \R$ with $|A|=k$ and $|A+A|\leq 2k-1+b\leq 3k-4$, then $A$ is contained in a $(k+b)$-term arithmetic progression. 
\end{theorem}

\begin{theorem}\cite[p. 15]{Frei73}\label{Fclass2} If $k>6$ and $A\subseteq \R$ with $|A|=k$ and $|A+A|\leq 3k-3$, then $A$ is contained in a $(2k-1)$-term arithmetic progression, or $A$ is a union of two arithmetic progressions of the same step size.
\end{theorem}

\begin{rem} In \cite[p. 23]{Frei73}, Freiman also claims a classification for $|A+A|=3k-2$. However, the proof is omitted, and as observed by Jin \cite{Jin07}, the classification is in fact incomplete. We return to this following Lemma \ref{925sum}.
\end{rem}

\noindent Theorems \ref{Fclass1} and \ref{Fclass2} can be applied on the product side to conclude that a set with very few products is either contained in a single geometric progression or is a union of two geometric progressions with the same common ratio $r$. The remainder of the proof of Theorem \ref{oldSP} consists of lower bounds on the number of sums determined by such sets, including the following fact that follows from the rational root theorem.

\begin{lemma}\cite[Lemma 2.4]{SP2023} \label{sg} If $A\subseteq \R$ is contained in a geometric progression with common ratio $r\in \Q$, $r\neq -2$, then $A$ is a Sidon set. In particular, if $|A|=k$, then $|A+A|=(k^2+k)/2$.
\end{lemma}

As remarked in \cite[Section 4]{SP2023}, Theorems \ref{Fclass1} and \ref{Fclass2} sufficiently pave the way for exact values of $SP(k)$ for $k\leq 9$, because hypothetical improvements on observed examples fall under the classification umbrella. In contrast, for $k=10$, the best observed example was $A=\{1,2,3,4,6,8,9,12,16,18\}$, with $|A+A|=30$ and $|AA|=29$, leading to the conjecture that $SP(10)=30$. To prove this, one must rule out the existence of a set of $10$ positive integers with at most $29$ sums and at most $29$ products, but since $29=3(10)-1$, Theorems \ref{Fclass1} and \ref{Fclass2} are not directly applicable, and new techniques are required. Here we achieve this goal and extend it one step further.

\begin{theorem} \label{mainNew} We have the exact values $SP(10)=30$ and $SP(11)=34$.  Further, the sets achieving $\max\{|A+A|,|AA|\}=SP(k)$ for $k=10,11$, provided in Table \ref{UBtable}, are unique up to scaling.
\end{theorem}

\noindent Since $SP(k)$ is defined as a minimum, we can establish the claimed upper bounds via one example for each $k$, as listed below. We encourage the reader to verify the values of $|A+A|$ and $|AA|$. The remainder of the paper is dedicated to establishing the corresponding lower bounds.

\begin{table}[H] 

\centering

\caption{Examples showing upper bounds for $SP(k)$ for $4\leq k \leq 11$.}
		\label {UBtable}


\begin{tabular}{||c||c||c||c||}

\hline

$k$ & $A$ & $|A+A|$ & $|AA|$\\

\hline\hline

4 & \{1, 2, 3, 4\} & 7 & 9 \\

\hline

5 & \{1, 2, 3, 4, 6\} & 10 & 12 \\

\hline

6 & \{1, 2, 3, 4, 6, 8\} & 13 & 15 \\

\hline

7 & \{1, 2, 3, 4, 6, 8, 12\} & 18 & 18 \\

\hline

8 & \{1, 2, 3, 4, 6, 8, 9, 12\} & 20 & 22 \\

\hline 

9 & \{1, 2, 3, 4, 6, 8, 9, 12, 16\} & 25 & 25 \\

\hline

10 & \{1, 2, 3, 4, 6, 8, 9, 12, 16, 18\} & 30 & 29 \\

\hline

11 & \{1, 2, 3, 4, 6, 8, 9, 12, 16, 18, 24\} & 34 & 32 \\

\hline

\end{tabular}
\end{table}

\section{Proof outline} \label{outline}
Suppose $A\subseteq \N$ with $|A|=10$ (resp. $11$) and $|AA|\leq 29$ (resp. $33$).  We show that $|A+A|\geq 31$ (resp. $35$) unless $A$ is, up to scaling, $\{1, 2, 3, 4, 6, 8, 9, 12, 16, 18\}$ (resp. $\{1, 2, 3, 4, 6, 8, 9, 12, 16, 18, 24\}$), as follows.

\begin{enumerate} \item Scale $A$ so that $A\subseteq \Q$ with $\min(A)=1$. Let $r$ be the least element of $A$ besides $1$, and let $L=\{\log_r(a): a\in A\}$, so the two least elements of $L$ are $0$ and $1$.

\

\item If at least $8$ elements of $L$ are contained in a single arithmetic progression, then at least $8$ elements of $A$ are contained in a single geometric progression, hence $|A+A|\geq (8^2+8)/2 =36$ by Lemma \ref{sg}.

\

\item In Section \ref{spsc}, aided by Theorem \ref{Fclass2} and Python computation, we establish that if $L$ does not contain at least $8$ elements of a single arithmetic progression, then $L=\{m+n\alpha: (m,n)\in P\}$, where $\alpha \notin \Q$ and $P\subseteq \Z^2$ belongs to a short list, up to translation and invertible linear transformation, and hence $A$ has one of a short list of two-dimensional geometric progression structures. In particular, up to scaling, at least $9$ elements of $A$ are contained in one of $G_1=\{r^ms^n: 0\leq m \leq 7, n\in \{0,1\}\}$ or $G_2=\{r^ms^n: 0\leq m \leq 3, 0\leq n \leq 2\}$,  with $r,s\in \Q$, $r,s>1$, and $\log_r(s)\notin \Q$.

\

\item In Section \ref{ccs}, using Python computation to search for solutions to the appropriate systems of bivariate polynomial equations, we control the number of nontrivial repeated sums, which we refer to as \textit{collisions} (also called \textit{additive quadruples}, which include trivial repeats), in subsets of two-dimensional geometric progressions like $G_1$ or $G_2$. In particular,  we show that a $9$-element subset of $G_1$ or $G_2$ determines at least $39$ distinct sums, unless $(r,s)$ lies in a list of $155$ exceptional pairs for $G_1$ or $75$ exceptional pairs for $G_2$.

\

\item Separately checking all exceptional $(r,s)$ pairs from Section \ref{ccs} for all two-dimensional geometric progression structures enumerated in Section \ref{spsc}, we find that all determine at least $31$ (resp. $35$) sums, except scalings of $\{1, 2, 3, 4, 6, 8, 9, 12, 16, 18\}$ and  $\{1, 2, 3, 4, 6, 8, 9, 12, 16, 18, 24\}$, respectively.

\

\item In Section \ref{usec}, to establish uniqueness (up to scaling), we outline a similar approach to items (1)-(5), but under the weaker assumption $|AA|\leq 30$ (resp. $|AA|\leq 34$), with less detailed classification and more ad hoc computation.
 
\end{enumerate}

\noindent Throughout the paper, we provide pseudocode desriptions of algorithms, and summaries of computational output. For those interested, all Python code and collected output is available at \url{https://github.com/Alex-Rice-Millsaps/Sum-Product-Summer-2025} in the form of annotated Jupyter notebooks.

\section{Small product set classification} \label{spsc}

To motivate our first computation, suppose $L\subseteq \R$ with $|L|=10$ and $|L+L|\leq 29$, suppose the two least elements of $L$ are $0$ and $1$, a property we refer to as \textit{reduced}, and let $L'=L\setminus\{0\}$. If $L$ is not contained in $\Z$, then $|L+L|\geq |L'+L'|+3$, because the addition of $0$ creates at least three news sums: $0$, $1$, and the least noninteger element of $L'$. Further, if $|L+L|\geq |L'+L'|+4$, then $|L'|=9$ and $|L'+L'|\leq 25$, which brings us one step closer to the applicability of Theorem \ref{Fclass2}. For these reasons, we focus on the case where $L$ is not contained in a single arithmetic progression, and $|L+L|=|L'+L'|+3$. In particular, letting $\alpha$ be the least noninteger element of $L$, we have that that all elements of $L$ take the form $m+n\alpha$ for integers $m,n\geq 0$. Otherwise, if $\beta$ is the least element of $L$ to not take this form, then $0,1,\alpha,\beta$ are four new sums created by adding $0$ to $L'$. Further, $\alpha$ must be irrational, as if $\alpha=p/q$, then $L$ is contained in an arithmetic progression of step size $1/q$. Noting the map $m+n\alpha \to (m,n)$ is a group isomorphism into $\Z^2$, we focus on characterizing the structure of $P\subseteq \Z^2$ defined by $A=\{m+n\alpha: (m,n)\in P\}$. Specifically, we classify $P$ up to translation and invertible  linear transformation, which we collectively refer to as \textit{similarity}.    

Here we describe an algorithm WinnersSearch, which takes as input $k,M\in \N$ and outputs a list of all $P\subseteq \Z_{\geq 0}^2$ satisfying $|P|=k$, $(0,0),(1,0),(0,1)\in P$, $|P+P|\leq M$, and $|P+P|=|P'+P'|+3$, where $P'=P\setminus\{(0,0)\}$, up to similarity. We build each set $P$ by appending elements $(m,n)\in \Z_{\geq 0}^2$ in nondecreasing order of $m+n$ (i.e. $\ell_1$-norm), which we refer to as \textit{score}. Crucially the condition $|P+P|=|P'+P'|+3$ requires that every newly appended element $(m,n)$ must be equal to a sum of two previously appended elements, otherwise $(0,0),(1,0),(0,1),(m,n)$ are four elements of $(P+P)\setminus(P'+P')$. Therefore, we can keep track of the score of the newest elements, and the pairwise sums of previous elements grouped by score, to get a short list of `next element' candidates to loop through and recursively continue the search. A search branch terminates when the number of pairwise sums exceeds $M$, and the branch is discarded, or when the set reaches $k$ elements with at most $M$ distinct sums, and $P$ is added to the list of outputs.

\begin{algorithm}[H]
\textbf{WinnersSearch} 
\begin{algorithmic}
\State{\textbf{Input:} $k,M\in \N$}

\State Winners $\gets \{\}$
\Function{Reflect}{$A$}
	\State \Return $\{(y,x): (x,y)\in A\}$
\EndFunction
\Function{Search}{$A$, ScoreDict, ScoreDictKeys, SumSet, Used, $t$, $k$, $M$}
    \For{$j\in $ ScoreDictKeys with $j\geq t$}
        \For{$x\in$ ScoreDict$(j)$}
            \If{$x \notin A$}
                \State $A \gets A\cup\{x\}$
                \If{$A \notin$ Used and \Call{Reflect}{$A$} $\notin$ Used}
                    \State Used $\gets$ Used $\cup \ \{A\}$
                    \For{$y\in A$}
						\State $z = (z_0,z_1)\gets x+y$
						\If{$z\notin$ SumSet}
							\State SumSet$\gets$ SumSet $\cup \ \{z\}$
							\State $q \gets z_0+z_1$
							\If{$q \notin$ ScoreDictKeys}
								\State ScoreDictKeys $\gets$ ScoreDictKeys$\ \cup \ \{q\}$
								\State ScoreDict$(q)\gets \{z\}$
							\Else
								\State ScoreDict$(q) \gets $ScoreDict$(q)\cup\{z\}$
							\EndIf
						\EndIf 
                    \EndFor
                    \If{$|A|=k$}
						\If{$|$SumSet$|\leq M$ and $A\notin$ Winners and \Call{Reflect}{$A$} $\notin$ Winners}
						\State Winners $\gets$ Winners $\cup \ \{A\}$
					
						\EndIf
                  
					\Else
						\If{$|$SumSet$|\leq M$}
							\State \Call{Search}{$A$, ScoreDict, ScoreDictKeys, SumSet, Used, $j$, $M$}
						\EndIf
                  
					\EndIf
                \EndIf
			\EndIf
        \EndFor
    \EndFor
\EndFunction
\State $A \gets \{(0,0), (0,1), (1,0)\}$ 
\State ScoreDict $\gets$ empty dictionary
\State ScoreDict$(0) \gets \{(0,0)\}$
\State ScoreDict$(1) \gets \{(0,1),(1,0)\}$
\State ScoreDict$(2) \gets \{(0,2),(1,1),(2,0)\}$
\State ScoreDictKeys$\gets \{0,1,2\}$
\State SumSet $\gets \{(0,0), (0,1), (1,0), (0,2),(1,1),(2,0)\}$
\State Used $\gets \{A\}$
\State \Call{Search}{$A$, ScoreDict, ScoreDictKeys, SumSet, Used, $2$, $k$, $M$}
\State \Return Winners

\end{algorithmic}
\end{algorithm}
\newpage
Analysis of WinnersSearch$(10,29)$ (runtime $25$ seconds) yields the following classification.

\begin{lemma}\label{1029sum} If $L\subseteq \R$ is reduced with $|L|=10$ and $|L+L|\leq 29$, then at least one of the following holds:
\begin{enumerate}[(1)] \item $L$ is contained in an arithmetic progression,
\item $|L+L|\geq |L'+L'| + 4$, where $L'=L\setminus\{0\}$,
\item $L=\{m+n\alpha: (m,n)\in P\},$ where $\alpha\notin \Q$ and $P$ takes one of the following forms up to similarity:
\begin{enumerate}[(i)] \item $\{(0,0),(1,0),\dots,(k_1-1,0),(0,1),(1,1)\dots,(k_2-1,1)\}, \  k_1\geq k_2, \ k_1+k_2=10,$

\item $\{(0,0),\dots,(k_1-1,0),(0,1),\dots, (k_2-2,1), (k_2,1)\}, \ k_1+k_2=10, $

\item  $\{(0,0),\dots,(k_1-1,0),(0,1),\dots,(k_2-1,1), (0,2),\dots,(k_3-1,2)\}$, \\$(k_1,k_2,k_3)\in \{(6,3,1),(5,4,1),(5,3,2),(4,4,2),(4,3,3),(3,4,3)\},$
\item $\{(0,0),(1,0),(2,0),(3,0),(0,1),(1,1),(2,1),(0,2),(1,2),(0,3)\}$.
\item $\{(0,0),(1,0),\dots, (k_1-1,0),(0,1),\dots,(k_2-1,1),(1,2),\dots,(k_3,2)\}$,  \\$(k_1,k_2,k_3)\in \{(5,4,1),(4,4,2),(3,4,3)\},$
\item $\{(0,0),(1,0),\dots,(6,0),(0,1),(2,1),(4,1)\}$.
\end{enumerate}

\end{enumerate}
\end{lemma}

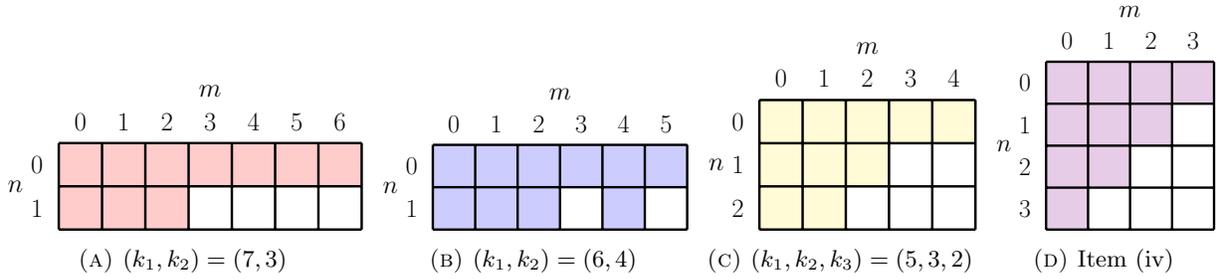
\begin{figure}[H]
\centering

\begin{subfigure}[t]{.29\textwidth}
\centering
\resizebox{\linewidth}{!}{%
\begin{tikzpicture}
	\fill[red!20] (0,2) rectangle (7,1);
    \fill[red!20] (0,1) rectangle (3,0);
	
    \draw[step=1cm, black, ultra thick] (0,2) grid (7,0)
    (0.5,2.5) node {\huge $0$}
    (1.5,2.5) node {\huge $1$}
    (2.5,2.5) node {\huge $2$}
    (3.5,2.5) node {\huge $3$}
    (4.5,2.5) node {\huge $4$}
    (5.5,2.5) node {\huge $5$}
    (6.5,2.5) node {\huge $6$}
    (-0.5,1.5) node {\huge $0$}
    (-0.5,0.5) node {\huge $1$};

    \node at (3.5,3.2) {\huge $m$};
    \node at (-1,1) {\huge $n$};

\end{tikzpicture}}
\caption{$(k_1,k_2)=(7,3)$}
\end{subfigure}\hspace{2mm}%
\begin{subfigure}[t]{.25\textwidth}
\centering
\resizebox{\linewidth}{!}{%
\begin{tikzpicture}
	\fill[blue!20] (0,2) rectangle (6,1);
    \fill[blue!20] (0,1) rectangle (3,0);
    \fill[blue!20] (4,1) rectangle (5,0);
    \draw[step=1cm, black, ultra thick] (0,2) grid (6,0)
    (0.5,2.5) node {\huge $0$}
    (1.5,2.5) node {\huge $1$}
    (2.5,2.5) node {\huge $2$}
    (3.5,2.5) node {\huge $3$}
    (4.5,2.5) node {\huge $4$}
    (5.5,2.5) node {\huge $5$}
    (-0.5,1.5) node {\huge $0$}
    (-0.5,0.5) node {\huge $1$};

    \node at (3,3.2) {\huge $m$};
    \node at (-1,1) {\huge $n$};

\end{tikzpicture}}
\caption{$(k_1,k_2)=(6,4)$}
\end{subfigure}\hspace{2mm}%
\begin{subfigure}[t]{.22\textwidth}
\centering
\resizebox{\linewidth}{!}{%
\begin{tikzpicture}
	\fill[yellow!20] (0,2) rectangle (5,1);
    \fill[yellow!20] (0,1) rectangle (3,0);
	\fill[yellow!20] (0,0) rectangle (2,-1);
    \draw[step=1cm, black, ultra thick] (0,2) grid (5,-1)
    (0.5,2.5) node {\huge $0$}
    (1.5,2.5) node {\huge $1$}
    (2.5,2.5) node {\huge $2$}
    (3.5,2.5) node {\huge $3$}
    (4.5,2.5) node {\huge $4$}
    (-0.5,1.5) node {\huge $0$}
    (-0.5,0.5) node {\huge $1$}
    (-0.5,-0.5) node {\huge $2$};

    \node at (2.5,3.2) {\huge $m$};
    \node at (-1,0.5) {\huge $n$};

\end{tikzpicture}}
\caption{$(k_1,k_2,k_3)=(5,3,2)$}
\end{subfigure}\hspace{2mm}%
\begin{subfigure}[t]{.18\textwidth}
\centering
\resizebox{\linewidth}{!}{%
\begin{tikzpicture}
	\fill[violet!20] (0,4) rectangle (4,3);
    \fill[violet!20] (0,3) rectangle (3,2);
    \fill[violet!20] (0,2) rectangle (2,1);
    \fill[violet!20] (0,1) rectangle (1,0);
    \draw[step=1cm, black, ultra thick] (0,4) grid (4,0)
    (0.5,4.5) node {\huge $0$}
    (1.5,4.5) node {\huge $1$}
    (2.5,4.5) node {\huge $2$}
    (3.5,4.5) node {\huge $3$}
    (-0.5,3.5) node {\huge $0$}
    (-0.5,2.5) node {\huge $1$}
    (-0.5,1.5) node {\huge $2$}
    (-0.5,.5) node {\huge $3$};

    \node at (2,5.2) {\huge $m$};
    \node at (-1,2) {\huge $n$};

\end{tikzpicture}}
\caption{Item (iv)}
\end{subfigure}

\caption{One example from items (i)-(iii), and unique structure from (iv), in Lemma \ref{1029sum}.}
\label{fig1}
\end{figure}

\begin{figure}[H]
\centering
\hfill
\begin{subfigure}[t]{0.22\textwidth}
\centering
\resizebox{\linewidth}{!}{%
\begin{tikzpicture}
	\fill[red!20] (0,2) rectangle (3,1);
    \fill[red!20] (0,1) rectangle (4,0);
	\fill[red!20] (1,0) rectangle (4,-1);
    \draw[step=1cm, black, ultra thick] (0,2) grid (4,-1)
    (0.5,2.5) node {\huge $0$} 
    (1.5,2.5) node {\huge $1$} 
    (2.5,2.5) node {\huge $2$} 
    (3.5,2.5) node {\huge $3$} 
    (-0.5,1.5) node {\huge $0$} 
    (-0.5,0.5) node {\huge $1$} 
    (-0.5,-0.5) node {\huge $2$};

    \node at (2,3.2) {\huge $m$};
    \node at (-1,0.5) {\huge $n$};

\end{tikzpicture}}
\caption{$(k_1,k_2,k_3)=(3,4,3)$}
\end{subfigure}\hfill
\begin{subfigure}[t]{0.32\textwidth}
\centering
\resizebox{\linewidth}{!}{%
\begin{tikzpicture}
	\fill[blue!20] (0,2) rectangle (7,1);
    \fill[blue!20] (0,1) rectangle (1,0);
    \fill[blue!20] (2,1) rectangle (3,0);
    \fill[blue!20] (4,1) rectangle (5,0);
    \draw[step=1cm, black, ultra thick] (0,2) grid (7,0)
    (0.5,2.5) node {\huge $0$} 
    (1.5,2.5) node {\huge $1$} 
    (2.5,2.5) node {\huge $2$} 
    (3.5,2.5) node {\huge $3$} 
    (4.5,2.5) node {\huge $4$}
    (5.5,2.5) node {\huge $5$}
    (6.5,2.5) node {\huge $6$}
    (-0.5,1.5) node {\huge $0$} 
    (-0.5,0.5) node {\huge $1$};

    \node at (3.5,3.2) {\huge $m$};
    \node at (-1,1) {\huge $n$};

\end{tikzpicture}}
\caption{Item (vi)}
\end{subfigure}\hfill \
\caption{Example from item (v), and unique structure from (vi), in Lemma \ref{1029sum}.}
\label{fig2}
\end{figure}
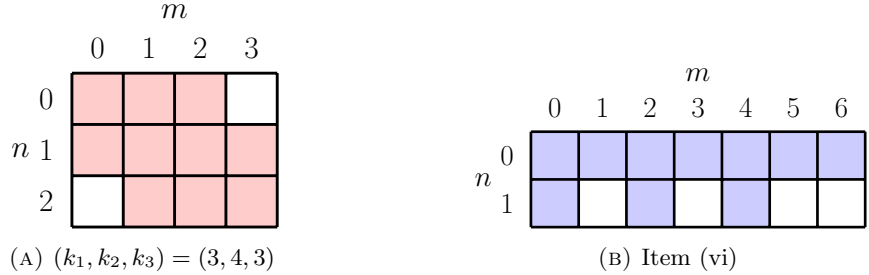

\noindent To account for item (2) in Lemma \ref{1029sum}, we also analyze WinnersSearch$(9,25)$ (runtime $6$ seconds).

\begin{lemma}\label{925sum} If $L\subseteq \R$ is reduced with $|L|=9$ and $|L+L|\leq 25$, then at least one of the following holds:
\begin{enumerate}[(1)] \item $L$ is contained in an arithmetic progression,
\item $|L+L|\geq |L'+L'| + 4$, where $L'=L\setminus\{0\}$,
\item $L=\{m+n\alpha: (m,n)\in P\},$ where $\alpha\notin \Q$ and $P$ takes one of the following forms up to similarity:
\begin{enumerate}[(i)] \item $\{(0,0),(1,0),\dots,(k_1-1,0),(0,1),(1,1)\dots,(k_2-1,1)\}, \  k_1\geq k_2, \ k_1+k_2=9,$
\item $\{(0,0),(1,0),\dots,(6,0),(8,0),(0,1)\}$,
\item $\{(0,0),(1,0),\dots,(6,0),(0,1),(2,1)\}$,

\item  $\{(0,0),\dots,(k_1-1,0),(0,1),\dots,(k_2-1,1), (0,2),\dots,(k_3-1,2)\}, \\ (k_1,k_2,k_3)\in \{(5,3,1),(4,3,2),(3,3,3)\}$.
\end{enumerate}
\end{enumerate}
\end{lemma}

\begin{rem} Items (iii) and (iv) in Lemma \ref{925sum} are missing from Freiman's $3k-2$ characterization in \cite[p. 23]{Frei73}. Item (iii) generalizes to $\{(0,0),(1,0),\dots,(k-3,0),(0,1),(2,1)\}$, as noted by Jin \cite[Introduction]{Jin07}. \end{rem}

\noindent Note that for sets $L$ in item (2) in Lemma \ref{925sum}, $L'$ satisfies $|L'|=8$ and $|L'+L'| \leq 21 = 3(8)-3$, hence, by Theorem \ref{Fclass2}, $L'$ is either contained in a single arithmetic progression, or is a union of two arithmetic progressions of the same step size. To complete our classification, we separately analyze the ways in which one element can be added to a structure enumerated in Lemma \ref{925sum}, or two elements can be added to a union of two arithmetic progressions of the same step size, while maintaining at most $30$ distinct sums. Combining these with the structures enumerated in Lemma \ref{1029sum} yields the following.

\begin{corollary} \label{1029sumcor} If $L\subseteq \R$ is reduced with $|L|=10$ and $|L+L|\leq 29$, then at least one of the following holds:
\begin{enumerate}[(1)] \item At least $8$ elements of $L$ are contained in a single arithmetic progression,
\item $L=\{m+n\alpha: (m,n)\in P\},$ where $\alpha\notin \Q$ and $P$ takes one of the forms listed in item (3) of Lemma \ref{1029sum}.



\end{enumerate}
\end{corollary}

\noindent We observe that, for all structures enumerated in item (3) of Lemma \ref{1029sum}, $P$ contains at least $9$ elements of $\{(m,n): 0\leq m \leq 7, n\in \{0,1\}\}$ or at least $9$ elements of $\{(m,n): 0\leq m \leq 3, 0\leq n \leq 2\}$. While we do not provide a full classification up to similarity for $L\subseteq \R$ with $|L|=11$ and $|L+L|\leq 33$, we do note that all such similarity classes are either output by WinnersSearch$(11,33)$, or obtained by adding an additional element to a structure described in Corollary \ref{1029sumcor}. Analysis of WinnersSearch$(11,33)$ (runtime $3$ minutes) confirms the following result.

\begin{corollary} \label{sumcornew} If $L\subseteq \R$ with $|L|=10$ and $|L+L|\leq 29$, or $|L|=11$ and $|L+L|\leq 33$, then at least one of the following holds:
\begin{enumerate}[(1)] \item At least $8$ elements of $L$ are contained in a single arithmetic progression,
\item $L=\{m+n\alpha: (m,n)\in P\},$ where $\alpha\notin \Q$ and $P$, up to similarity, contains at least $9$ elements of $\{(m,n): 0\leq m \leq 7, n\in \{0,1\}\}$,
\item $L=\{m+n\alpha: (m,n)\in P\},$ where $\alpha\notin \Q$ and $P$, up to similarity, contains at least $9$ elements of $\{(m,n): 0\leq m \leq 3, 0\leq n \leq 2\}$.
\end{enumerate}
\end{corollary}

\begin{figure}[H]
\centering
\begin{tikzpicture}[scale=0.7, every node/.style={transform shape}]
    \draw[step=1cm, black, ultra thick] (0,2) grid (8,-1)
    (0.5,2.5) node [ultra thick, black] {\huge $0$} 
    (1.5,2.5) node [ultra thick, black] {\huge $1$} 
    (2.5,2.5) node [ultra thick, black] {\huge $2$} 
    (3.5,2.5) node [ultra thick, black] {\huge $3$} 
    (4.5,2.5) node [ultra thick, black] {\huge $4$}
    (5.5,2.5) node [ultra thick, black] {\huge $5$}
    (6.5,2.5) node [ultra thick, black] {\huge $6$}
    (7.5,2.5) node [ultra thick, black] {\huge $7$}
    (-0.5, 1.5) node [ultra thick, black] {\huge $0$} 
    (-0.5, 0.5) node [ultra thick, black] {\huge $1$}
	(-0.5, -0.5) node [ultra thick, black] {\huge $2$} 
    ;

    \node at (4,3.2) {\huge $m$};
    \node at (-1,0.5) {\huge $n$};
    \draw[dashed, red, ultra thick] (-.125,2.125) rectangle (8.125,-0.125);
	\draw[dashed, blue, ultra thick] (-.125,2.125) rectangle (4.125,-1.125);
\end{tikzpicture}
\caption{The red and blue boxes indicate the configurations from items (2) and (3), respectively, in Corollary \ref{sumcornew}.}
\label{fig3}
\end{figure}
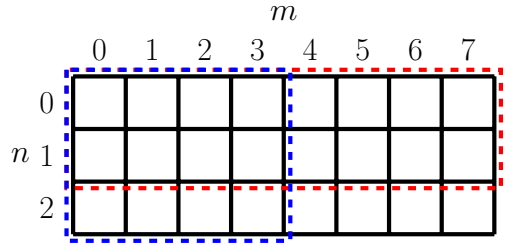

\noindent Applying Corollary \ref{sumcornew} on the product side yields the final output of this section.

\begin{corollary} \label{1029prodcor2} If $A\subseteq (0,\infty)$ with $|A|=10$ and $|AA|\leq 29$, or $|A|=11$ and $|AA|\leq 33$, then at least one of the following holds:
\begin{enumerate}[(1)] \item At least $8$ elements of $A$ are contained in a single geometric progression,
\item Up to scaling, at least $9$ elements of $A$ are contained in $\{r^ms^n: 0\leq m \leq 7, n\in \{0,1\}\}$ with $r,s>1$, $\log_r(s)\notin \Q$.
\item Up to scaling, at least $9$ elements of $A$ are contained in  $\{r^ms^n: 0\leq m \leq 3, 0\leq n \leq 2\}$ with $r,s>1$, $\log_r(s)\notin \Q$.
\end{enumerate}
\end{corollary}


\begin{proof} Suppose $A\subseteq (0,\infty)$ with $|A|=10$ and $|AA|\leq 29$, or $|A|=11$ and $|AA|\leq 33$. After scaling, assume $1=\min(A)$, and let $r>1$ be the least element of $A$ besides $1$. Let $L=\{\log_r(x): x\in A\}$, so $|L|=10$ and $|L+L|\leq 29$, or $|L|=11$ and $|L+L|\leq 33$ and the two least elements of $L$ are $0$ and $1$. 

\noindent If at least $8$ elements of $L$ are contained in a single arithmetic progression, then at least $8$ elements of $A$ are contained in a single geometric progression. Otherwise, $L$ has property (2) or (3) from Corollary \ref{sumcornew}, with the corresponding set of ordered pairs $P$ requiring at most translation and invertible linear transformation. Letting $s=r^{\alpha}$, the classification follows, noting that translation of $P$ corresponds to scaling $A$, while invertible linear tranformation of $P$ corresponds to changing $r$ and $s$. 
\end{proof}

\begin{rem} The example in Table \ref{UBtable} for $k=10$ corresponds to item (iii) in Lemma \ref{1029sum} with $r=2$, $s=3$, and $(k_1,k_2,k_3)=\{5,3,2\}$. The example for $k=11$ is the same except $k_2=4$.
\end{rem}

\section{Collision control} \label{ccs}

In this section, we provide lower bounds on the number of distinct sums determined by subsets of $G_1=\{r^ms^n: 0\leq m \leq 7, n\in \{0,1\}\}$ and $G_2=\{r^ms^n: 0\leq m \leq 3, 0\leq n \leq 2\}$, with $r,s\in \Q$, $r,s>1$, and $\log_r(s)\notin \Q$. We take an approach that is similar in spirit to \cite[Section 3.1]{SP2023}, but is more purpose-built and computationally aided. To give the idea, suppose $A\subseteq G_1$ with $|A|=k$, and let $B=A\cap\{1,r,\dots,r^7\}$ and $C=A\cap\{s,rs,\dots,r^7s\}$. We know from Lemma \ref{sg} that $B$ and $C$ are Sidon sets, so $|A+A|$ is only drawn away from its maximum value $(k^2+k)/2$ by intersections amongst $B+B,C+C,$ and $B+C$, as well as elements with multiple representations in $B+C$. In  \cite[Lemma 3.5]{SP2023}, it is shown that each of these four types of \textit{two-row collisions} are restricted to a single geometric family. To clarify, consider $(B+B)\cap (B+C)$, elements of which correspond to solutions to the equation $r^a+r^b=r^c+sr^d$, and every such solution can be translated (meaning scaled by a power of $r$) to make $\min(a,b,c,d)=0$. The approach of \cite[Lemma 3.5]{SP2023} is to show, via the rational root theorem, that there is only one such solution with $\min(a,b,c,d)=0$, except when $r=2$. In nonexceptional cases, all other solutions are translations of that one, which limits the total number of solutions to the number of  translations. Repeating this process for all four collision types, and adding together the number of translations,  yields \cite[Corollary 3.7]{SP2023}.

Here we sharpen this approach in two ways, and then extend it in one. First, we keep track not just of exceptional values of $r$, but of all exceptional $(r,s)$ pairs, so that we can then check them all individually after the fact. Further, we not only establish that each collision type is restricted to a single geometric family, but also that only one of the four can actually occur for a given set $A$, except in a small family of special cases, which we call \textit{double collisions}, corresponding to the system $1+r^3=s+rs$, $1+r^2=r+s$, satisfied when $s=r^2-r+1$. Therefore, $\frac{k^2+k}{2}-|A+A|$ is bounded by the number of translations of one solution in the typical case, or a specific pair of solutions in the double collision case. See Figure \ref{colfigure} below for illustrations in which $\frac{k^2+k}{2}-|A+A|$ is $4$ and $7$, respectively.  Finally, we repeat the approach for $G_2$ and investigate all potential types of \textit{three-row collisions} between $\{1,r,r^2,r^3\}$, $\{s,rs,r^2s,r^3s\}$, and $\{s^2,rs^2,r^2s^2,r^3s^3\}$, arriving at analogous conclusions. We achieve these goals via Python computation, and as a result our conclusions are limited to configurations within $G_1$ and $G_2$, rather than more general two-dimensional geometric progressions. 

\begin{figure}[H]
\captionsetup[subfigure]{font=footnotesize}
\centering 
\hfill 
\begin{subfigure}[t]{0.25\textwidth}
\centering
\resizebox{\linewidth}{!}{%
\begin{tikzpicture}

    \fill[blue!20] (0,2) rectangle (5,1);
    \fill[blue!20] (0,1) rectangle (4,0);

    \draw[step=1cm, black, ultra thick] (0,2) grid (5,0)
        (0.5,2.45) node [ultra thick, black] {\huge $1$} 
        (1.5,2.38) node [ultra thick, black] {\huge $r$} 
        (2.5,2.5) node [ultra thick, black] {\huge $r^2$} 
        (3.5,2.5) node [ultra thick, black] {\huge $r^3$} 
        (4.5,2.5) node [ultra thick, black] {\huge $r^4$}
        (-0.5,1.5) node [ultra thick, black] {\huge $1$} 
        (-0.5,0.5) node [ultra thick, black] {\huge $s$};

    \coordinate (c11) at (0.5,1.5);   
    \coordinate (cr31) at (3.5,1.7);  
    \coordinate (c1s) at (0.5,0.5);   
   
    \coordinate (crs) at (1.5,0.3);   
    \coordinate (cr21) at (2.5,1.5);  
    \coordinate (cr1) at (1.5,1.5);   

    \draw[red, thick] (c11) -- (c1s); 
    \draw[->, red, thick, out=130, in=50, distance=0.7cm] (cr1) to (cr1);

    \foreach \pt in {c11,cr1,c1s} {
        \fill[black] (\pt) circle (1.2pt);
    }

\end{tikzpicture}}
\caption{Single collision: $1+s=2r$, which has  four translations.}
\end{subfigure}\hfill
\begin{subfigure}[t]{0.29\textwidth}
\centering
\resizebox{\linewidth}{!}{%
\begin{tikzpicture}

    \fill[yellow!20] (0,2) rectangle (6,1);
    \fill[yellow!20] (0,1) rectangle (4,0);

    \draw[step=1cm, black, ultra thick] (0,2) grid (6,0)
        (0.5,2.45) node [ultra thick, black] {\huge $1$} 
        (1.5,2.38) node [ultra thick, black] {\huge $r$} 
        (2.5,2.5) node [ultra thick, black] {\huge $r^2$} 
        (3.5,2.5) node [ultra thick, black] {\huge $r^3$} 
        (4.5,2.5) node [ultra thick, black] {\huge $r^4$}
        (5.5,2.5) node [ultra thick, black] {\huge $r^5$}
        (-0.5,1.5) node [ultra thick, black] {\huge $1$} 
        (-0.5,0.5) node [ultra thick, black] {\huge $s$};

    \coordinate (c11) at (0.5,1.5);   
    \coordinate (c11a) at (0.5,1.7);   
    \coordinate (cr31) at (3.5,1.7);  
    \coordinate (c1s) at (0.5,0.5);   
    \coordinate (c1sa) at (0.5,0.3);   
    \coordinate (crs) at (1.5,0.3);   
    \coordinate (cr21) at (2.5,1.5);  
    \coordinate (cr1) at (1.5,1.5);   
    \coordinate (cr1a) at (1.5,1.3);   

    \draw[red, thick] (c11a) -- (cr31); 
    \draw[red, thick] (crs) -- (c1sa);  

    \draw[blue, thick] (c11) -- (cr21); 
    \draw[blue, thick] (cr1a) -- (c1s);  

    \foreach \pt in {c11,c11a,cr31,c1s,c1sa,crs,cr21,cr1a} {
        \fill[black] (\pt) circle (1.2pt);
    }

\end{tikzpicture}}
\caption{Double collision: $1+r^3=s+rs$ (red) has three translations, while  $1+r^2=r+s$ (blue) has four.}
\end{subfigure}\hfill \
\caption{Single and double two-row collisions.}
\label{colfigure}
\end{figure}
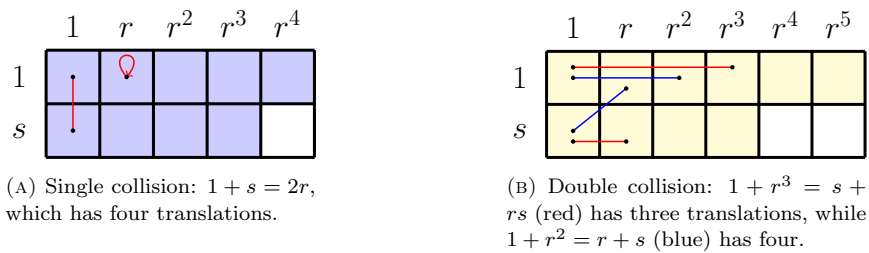

\newpage

\subsection{Collision control computations} Every two-row collision family in $G_1$ corresponds to a solution of the form $p(r,s)=r^a+s^tr^b-s^ur^c-s^vr^d=0$, where $a,b,c,d\in \{0,\dots,7\}$ with $\min(a,b,c,d)=0$ and $(t,u,v)\in T_1=\{(0,0,1),(0,1,1),(1,0,1),(1,1,1)\}$. Further, some nonredundancy and nondegeneracy conditions can be imposed, for example if $t=0$, we can assume $a\geq b$ by commutativity, while if $u=0$, we can assume $a\neq c$ because $s$ is not a rational power of $r$. We want to  determine all $(r,s)$ pairs for which multiple collision types can occur simultaneously, or two families of the same type can occur simultaneously. We index these checks with tuples $(t,u,v,x,y,z)$, defining $p(r,s)=r^a+s^tr^b-s^ur^c-s^vr^d$ and $q(r,s)=r^e+s^xr^f-s^yr^g-s^zr^h$, and look for solutions to the system $p(r,s)=q(r,s)=0$. This allows additional filters to be imposed, such as if $(t,u,v)=(x,y,z)$ we can ignore cases where $(a,b,c,d)=(e,f,g,h)$, and if $(0,t)=(u,v)=(0,x)=(y,z)$, we can ignore cases when $(a,b,c,d)=(g,h,e,f)$. We loop through exponent tuples $(t,u,v,x,y,z)$ with $(t,u,v),(x,y,z)\in T_1$, then loop through octuples $(a,b,c,d,e,f,g,h)\in \{0,1,\dots,7\}^8$, imposing $\min(a,b,c,d)=\min(e,f,g,h)=0$ and all nonredundancy and nondegeneracy conditions through a function GoodOctuple$(a,b,c,d,e,f,g,h,t,u,v,x,y,z)$. For three-row collisions in $G_2$, we do a similar loop, but with $(t,u,v),(x,y,z)\in T_2=\{(0,1,1),(0,0,1),(1,1,1),(1,0,1),(0,2,2),(0,0,2),(2,2,2),(2,0,2),$ $(1,1,2), (2,1,1),(1,2,2),(2,1,2),(0,1,2),(1,0,2)\}$ and $(a,b,c,d,e,f,g,h)\in \{0,1,2,3\}^8$.

Recall that for $p(r,s)$ and $q(r,s)$, we can compute two separate polynomial \textit{resultants}: $R_1(r)=$res$_s(p,q)$ is a polynomial in $r$ only with the property that $R_1(r_0)=0$ if and only if there exists $s_0$ such that $p(r_0,s_0)=q(r_0,s_0)=0$, while $R_2(s)=$res$_r(p,q)$ is an analogous polynomial in $s$ only. If these resultants are nonzero, then there are only finitely many pairs $(r,s)$ satisfying $p(r,s)=q(r,s)=0$, which can be added to our list of exceptions if $r,s\in \Q$, $r,s>1$, and $\log_r(s)\notin \Q$. The resultants are each $0$ if and only if $p$ and $q$ share a common factor, which is a potentially harmful scenario. However, these common factors may be spurious for our purposes. For example, if $\gcd(p,q)=r-s$, then the infinitely many simultaneous solutions come when $r=s$, which we forbid. Other ignorable instances include $\gcd$'s with no rational roots, or that force one of $r$ or $s$ to be negative. The following algorithm desribes how we collect all possible $\gcd$'s for all two and three-row collisions at once. We run PossibleGCDs$(8,E_1)$ (runtime 15 hours), where $E_1$ encodes pairs of triples in $T_1$, and PossibleGCDs$(4,E_2)$ (runtime $5$ hours), where $E_2$ encodes pairs of triples in $T_2$. The output consists of ignorable $\gcd$'s, except for ones corresponding to double collisions as in Figure \ref{colfigure}(B).

\begin{algorithm}[H]
\textbf{PossibleGCDs}
\begin{algorithmic}
\State{\textbf{Input:} $L \in \N$, ExponentTuples}
\State GCDdict $\gets$ empty dictionary
\For{$(t,u,v,x,y,z)\in$ ExponentTuples}
	\State GCDdict$(t,u,v,x,y,z) \gets \{\}$
	\For{$(a,b,c,d,e,f,g,h)\in\{0,\dots,L-1\}^8$}
		\If{\Call{GoodOctuple}{$a,b,c,d,e,f,g,h,t,u,v,x,y,z$}}
			\State $p\gets r^a+s^tr^b-s^ur^c-s^vr^d$
			\State $q\gets r^e+s^xr^f-s^yr^g-s^zr^h$
			\State $D \gets$ \Call{GCD}{$p,q$}
			\If{not \Call{IsConstant}{D}}
				\State GCDdict$(t,u,v,x,y,z) \gets$ GCDdict$(t,u,v,x,y,z)\cup\{D\}$
			\EndIf
		\EndIf
	\EndFor

\EndFor
\State \Return GCDdict
\end{algorithmic}
\end{algorithm}

We have now accounted for simultaneous collisions corresponding to common factors of $p(r,s)$ and $q(r,s)$, so when searching for solutions to the system $p(r,s)=q(r,s)=0$, we can first compute $D=\gcd(p,q)$ and replace $p$ and $q$ with $p/D$ and $q/D$, respectively, which will have nonzero resultants. The following algorithm collects instances of simultaneous collisions, given a row length and a list of collision type pairs.  

\begin{algorithm}[H]
\textbf{CollisionCheck}
\begin{algorithmic}
\State{\textbf{Input:} $L \in \N$, ExponentTuples}
\State CollisionDict $\gets$ empty dictionary
\For{$(t,u,v,x,y,z)\in$ ExponentTuples}
	\State CollisionDict$(t,u,v,x,y,z) \gets \{\}$
	\For{$(a,b,c,d,e,f,g,h)\in\{0,\dots,L-1\}^8$}
		\If{\Call{GoodOctuple}{$a,b,c,d,e,f,g,h,t,u,v,x,y,z$}}
			\State $p\gets r^a+s^tr^b-s^ur^c-s^vr^d$
			\State $q\gets r^e+s^xr^f-s^yr^g-s^zr^h$
			\State $D \gets$ \Call{GCD}{$p,q$}
			\State $p \gets p/D$
			\State $q \gets q/D$
			\State Res1 $\gets$ \Call{Resultant}{$p,q,s$}
			\State Roots1 $\gets$ \Call{RationalRoots}{Res1,$r$}
			\State Res2 $\gets$ \Call{Resultant}{$p,q,r$}
			\State Roots2 $\gets$ \Call{RationalRoots}{Res2,$s$}
			\For{$(r_0,s_0)\in$ Roots1 $\times$ Roots2 with $r_0,s_0>1$}
				\If{$p(r_0,s_0)=q(r_0,s_0)=0$}
					\State CollisionDict$(t,u,v,x,y,z) \gets$ CollisionDict$(t,u,v,x,y,z) \cup \{(r_0,s_0,a,b,c,d,e,f,g,h)\}$
				\EndIf
			\EndFor

			\EndIf
		
	\EndFor

\EndFor
\State \Return CollisionDict
\end{algorithmic}
\end{algorithm}

\subsection{Two-row collisions} Running CollisionCheck$(8,E_1)$ (runtime $16$ hours), where $E_1$ encodes all pairs of triples in $T_1$, then extracting all $(r,s)$ pairs from the resulting collision dictionary, yields the following.

\begin{lemma}\label{tworows} Suppose $r,s\in \Q$ with $r,s>1$ and $\log_r(s)\notin \Q$. Apart from $155$ exceptional $(r,s)$ pairs (all with $r\in \{2,3,3/2\}$), there are at most two total quadruples $(a,b,c,d)$ with $0\leq a,b,c,d \leq 7$ and $\min(a,b,c,d)=0$ satisfying any of the following equations with the given non-redundancy conditions:

\begin{enumerate}[(i)] \item $r^a+r^b= s(r^c+r^d)$, $a\geq b$, $c\geq d$
\item $r^a+r^b= r^c+sr^d$, $a\geq b$, 
\item $r^a+sr^b= s(r^c+r^d)$, $c\geq d$,
\item $r^a+sr^b= r^c+sr^d$, $a>c$,
\end{enumerate}
and no quadruple satisfies more than one of the four equations. Further, two total solutions only occur when $s=(r^{2\ell}-r^{\ell}+1)/r^j$, $\ell\in\{1,2\}$, $0\leq j <2\ell$, and the quadruples are $(3\ell,0,j+\ell,j)$, $(2\ell,0,\ell,j)$, solving equations (i) and (ii), respectively, or $s=r^{j}/(r^{2\ell}-r^{\ell}+1)$, $j\geq 2\ell$, and the quadruples are $(j+\ell,j,3\ell,0)$, $(j,\ell,2\ell,0)$, solving equations (i) and (iii), respectively. 
\end{lemma}

\noindent Note that all of the double collision cases enumerated in Lemma \ref{tworows} come from transformations of the aforementioned system $1+r^3=s+rs$, $1+r^2=r+s$ when $s=r^2-r+1$, via scaling by powers of $r$, taking reciprocals, or substituting $r^2$ for $r$. 

Returning to the discussion from the beginning of the section, now that we know collisions are restricted to one or two total geometric families in nonexceptional cases, we can analyze the amount of damage they can do by counting translations. Again letting $B=A\cap\{1,r,\dots,r^7\}$ and $C=A\cap\{s,rs,\dots,r^7s\}$, we note that is not possible for a collision to involve only one distinct element from each of $B$ and $C$, as $s$ is not equal to a rational power of $r$. In particular, for each single family of collisions, the number of translations is either bounded by $\min(|B|-1,|C|)$, if more than one element of $B$ is involved, or $\min(|B|,|C|-1)$. In other words, letting $\{k_1,k_2\}=\{|B|,|C|\}$ with $k_1\geq k_2$, the number of translations is bounded by $\min\{k_1-1,k_2\}$. Further, in the double collision case as in Figure \ref{colfigure}(B), the first family corresponds to $1+r^3=s+rs$, which involves two elements from each of $B$ and $C$, so the number of translations is bounded by $k_2-1$, while the second family corresponds to $1+r^2=r+s$, which involves three elements of $B$ and one element of $C$, so the number of translations is bounded by $\min(k_1-2,k_2)$. Putting all of these observations together yields the following specialized sharpening of \cite[Corollary 3.7]{SP2023}.

\begin{corollary} \label{2rowcor} Suppose $r,s\in \Q$ with $r,s>1$, $\log_r(s)\notin \Q$, and $(r,s)$ is not an exceptional pair as defined in Lemma \ref{tworows}. If  $A\subseteq\{r^ms^n: 0\leq m\leq 7, n\in \{0,1\}\}$ and $|A|=k\geq 3$, then $$|A+A|\geq \frac{k^2+k}{2}-\begin{cases}\min\{k-3,2k_2-1\} &\text{if double collision and }  k_2\geq 1 \\ \min\{k_1-1,k_2\} & \text{else}\end{cases},$$ where $\{k_1,k_2\}=\{|A\cap\{1,r,\dots,r^7\}|, |A\cap\{s,rs,\dots,r^7s\}|\}$, $k_1\geq k_2$. 
\end{corollary}

\noindent Applying Corollary \ref{2rowcor} in the worst-case scenario $k_1=\lceil k/2 \rceil$ yields the following more consise lower bounds, which represent a substantial improvement over the formulas provided in \cite[Theorem 3.3]{SP2023}.

\begin{corollary} \label{2rowcor2} Suppose $r,s\in \Q$ with $r,s>1$, $\log_r(s)\notin \Q$, and $(r,s)$ is not an exceptional pair as defined in Lemma \ref{tworows}. If  $A\subseteq\{r^ms^n: 0\leq m\leq 7, n\in \{0,1\}\}$ and $|A|=k\geq 3$, then $$|A+A|\geq \begin{cases} (k^2-k+6)/2 &\text{if double collision}  \\ (2k^2+3+(-1)^{k})/4 & \text{else}\end{cases}.$$  
\end{corollary}

\subsection{Three-row collisions} Running CollisionCheck$(4,E_2)$ (runtime 6 hours), where $E_2$ encodes all pairs of triples in $T_2$, then extracting all $(r,s)$ pairs from the resulting collision dictionary, yields the following.

\begin{lemma} \label{threerows} Suppose $r,s\in \Q$ with $r,s>1$ and $\log_r(s)\notin \Q$. Apart from $75$ exceptional $(r,s)$ pairs (all with $r\in \{2,3,4,5,3/2,5/2,4/3,5/3,5/4,9/5\}$), there are at most two quadruples $(a,b,c,d)$ with $0\leq a,b,c,d \leq 3$ and $\min(a,b,c,d)=0$ satisfying any of the following equations with the given non-redundancy conditions:

\begin{enumerate}[(i)] 
\item $r^a+r^b= s^2(r^c+r^d)$, $a\geq b$,$c\geq d$
\item $r^a+r^b= r^c+s^2r^d$, $a\geq b$,
\item $r^a+s^2r^b = s^2(r^c+r^d)$, $c\geq d$,
\item $r^a+s^2r^b = r^c + s^2r^d$, $a>c$
\item $r^a+r^b=sr^c+s^2r^d$, $a\geq b$,
\item $r^a+sr^b = r^c+s^2r^d$, 
\item $r^a+sr^b= sr^c+s^2r^d$, 
\item $r^a+s^2r^b=s(r^c+r^d)$, $c\geq d$
\item $r^a+sr^b=s^2(r^c+r^d)$, $c\geq d$
\item $r^a+s^2r^b=sr^c+s^2r^d$,

\end{enumerate} 
 and no quadruple satisfies more than one of the ten equations. Further, two total solutions only occur in the double collision case discussed in Lemma \ref{tworows}, with $\ell=1$ and $s$ replaced by $s^2$. Finally, if any of the four equations in Lemma \ref{tworows} have a solution, then the ten equations listed above do not.
\end{lemma}

\noindent Unlike the two-row case, collisions in $G_2=\{r^ms^n: 0\leq m \leq 3, 0\leq n \leq 2\}$ can be translated both horizontally or vertically, but the dimensions limit us to three horizontal choices (scaling by $1,r,$ or $r^2)$ and two vertical choices (scaling by $1$ or $s$). In other words, a single family of collisions deducts at most $6$ from the maximum value of $A+A$. Further, for a double collision, the family corresponding to $1+r^3=s+rs$ has one horizontal choices and two vertical choices, while the family corresponding to $1+r^2=r+s$ has two horizontal choices and two vertical choices, so again the maximum damage is $2+4=6$. We illustrate these worst-case scenarios in Figure \ref{colfigure2} and summarize the result in the following corollary. 

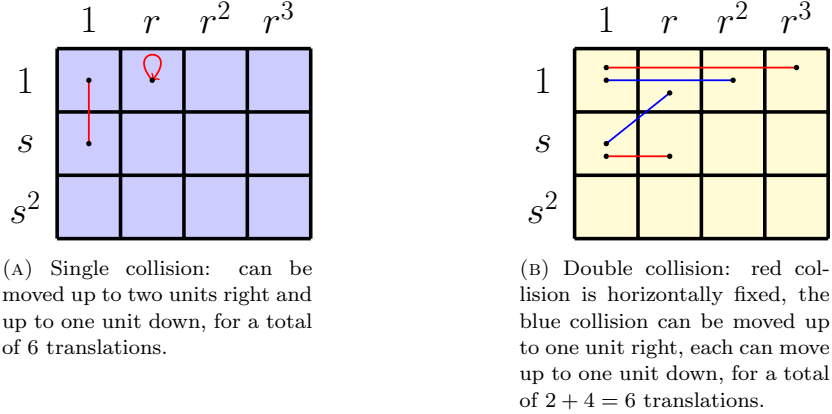
\begin{figure}[H]
\captionsetup[subfigure]{font=footnotesize}
\centering 
\hfill 
\begin{subfigure}[t]{0.25\textwidth}
\centering
\resizebox{\linewidth}{!}{%
\begin{tikzpicture}

    \fill[blue!20] (0,2) rectangle (4,-1);

    \draw[step=1cm, black, ultra thick] (0,2) grid (4,-1)
        (0.5,2.45) node [ultra thick, black] {\huge $1$} 
        (1.5,2.38) node [ultra thick, black] {\huge $r$} 
        (2.5,2.5) node [ultra thick, black] {\huge $r^2$} 
        (3.5,2.5) node [ultra thick, black] {\huge $r^3$} 
        (-0.5,1.5) node [ultra thick, black] {\huge $1$} 
        (-0.5,0.5) node [ultra thick, black] {\huge $s$}
		(-0.5,-0.5) node [ultra thick, black] {\huge $s^2$};

    \coordinate (c11) at (0.5,1.5);   
    \coordinate (cr31) at (3.5,1.7);  
    \coordinate (c1s) at (0.5,0.5);   
   
    \coordinate (crs) at (1.5,0.3);   
    \coordinate (cr21) at (2.5,1.5);  
    \coordinate (cr1) at (1.5,1.5);   

    \draw[red, thick] (c11) -- (c1s); 
    \draw[->, red, thick, out=130, in=50, distance=0.7cm] (cr1) to (cr1);

    \foreach \pt in {c11,cr1,c1s} {
        \fill[black] (\pt) circle (1.2pt);
    }

\end{tikzpicture}}
\caption{Single collision: can be moved up to two units right and up to one unit down, for a total of $6$ translations.}
\end{subfigure}\hfill
\begin{subfigure}[t]{0.25\textwidth}
\centering
\resizebox{\linewidth}{!}{%
\begin{tikzpicture}

    \fill[yellow!20] (0,2) rectangle (4,-1);

    \draw[step=1cm, black, ultra thick] (0,2) grid (4,-1)
        (0.5,2.45) node [ultra thick, black] {\huge $1$} 
        (1.5,2.38) node [ultra thick, black] {\huge $r$} 
        (2.5,2.5) node [ultra thick, black] {\huge $r^2$} 
        (3.5,2.5) node [ultra thick, black] {\huge $r^3$} 
       
        (-0.5,1.5) node [ultra thick, black] {\huge $1$} 
        (-0.5,0.5) node [ultra thick, black] {\huge $s$}
		(-0.5,-0.5) node [ultra thick, black] {\huge $s^2$};

    \coordinate (c11) at (0.5,1.5);   
    \coordinate (c11a) at (0.5,1.7);   
    \coordinate (cr31) at (3.5,1.7);  
    \coordinate (c1s) at (0.5,0.5);   
    \coordinate (c1sa) at (0.5,0.3);   
    \coordinate (crs) at (1.5,0.3);   
    \coordinate (cr21) at (2.5,1.5);  
    \coordinate (cr1) at (1.5,1.5);   
    \coordinate (cr1a) at (1.5,1.3);   

    \draw[red, thick] (c11a) -- (cr31); 
    \draw[red, thick] (crs) -- (c1sa);  

    \draw[blue, thick] (c11) -- (cr21); 
    \draw[blue, thick] (cr1a) -- (c1s);  

    \foreach \pt in {c11,c11a,cr31,c1s,c1sa,crs,cr21,cr1a} {
        \fill[black] (\pt) circle (1.2pt);
    }

\end{tikzpicture}}
\caption{Double collision: red collision is horizontally fixed, the blue collision can be moved up to one unit right, each can move up to one unit down, for a total of $2+4=6$ translations.}
\end{subfigure}\hfill \
\caption{Single and double three-row collisions, showing maximum of $6$ collisions in $G_2$.}
\label{colfigure2}
\end{figure}
 
\begin{corollary} \label{3rowcor} Suppose $r,s\in \Q$ with $r,s>1$, $\log_r(s)\notin \Q$, and $(r,s)$ is not an exceptional pair as defined in Lemma \ref{tworows} or Lemma \ref{threerows}. If $A\subseteq \{r^ms^n: 0\leq m \leq 3, 0\leq n \leq 2\}$ with $|A|=k$, then $$|A+A|\geq \frac{k^2+k}{2}-6.$$

\end{corollary} 

\noindent Applying Corollaries \ref{2rowcor} and \ref{3rowcor} with $k=9$, we find that every $9$-element subset of $G_1$ or $G_2$ determines at least $39$ distinct sums, apart from the stipulated exceptional $(r,s)$ pairs. To complete the proof, we 
collect all exceptional pairs from Lemmas \ref{tworows} and \ref{threerows}, plug them all in to each $10$ and $11$ element two-dimensional geometric progression configuration collected via the computations outlined in Section \ref{spsc}. As promised, we find all $10$-element sets determine at least $31$ distinct sums, and all $11$-element sets determine at least 35 distinct sums, except scalings of $\{1, 2, 3, 4, 6, 8, 9, 12, 16, 18\}$ and  $\{1, 2, 3, 4, 6, 8, 9, 12, 16, 18, 24\}$, respectively.

\section{Uniqueness} \label{usec}

While the arguments and computations in Sections \ref{spsc} and \ref{ccs} conclusively show $SP(10)=30$ and $SP(11)=34$, and the conclusion of Section \ref{ccs} includes a uniqueness statement, this is not quite enough to establish uniqueness of $A\subseteq \N$ with $|A|=10$ (resp. $11$) and $\max\{|A+A|,|AA|\}=30$ (resp. $34$). This shortcoming is because our initial assumption, as seen in the outline in Section \ref{outline}, is $|AA|\leq 29$ (resp. $33$), as that is required for a potential `dethronement' of our known examples. In particular, our work in Sections \ref{spsc} and \ref{ccs} does not rule out the existence of $A\subseteq \N$ with $|A|=10$ (resp. $11$), $|AA|=30$, (resp. $34$), and $|A+A|\leq 30$ (resp. $34$). Below, we quickly outline a repeat of our approach under the weaker assumption $|AA|\leq 30$ (resp. $34$). While we are less detailed in our classifications, particularly for $10$-element sets, the overall argument extends without obstruction.

\begin{enumerate} \item \label{7u} If $L\subseteq \R$ is reduced with $|L|=7$ and $|L+L|\leq 18=3(7)-3$, then by Theorem \ref{Fclass2}, $L$ is either contained in a a single arithmetic progression, or $L$ is a union of two arithmetic progressions of the same step size. In the latter case, we can assume $L=\{m+n\alpha: (m,n)\in P\}$ where $\alpha\notin \Q$ and $P=\{(0,0),\dots,(k_1-1,0),(0,1),\dots,(k_2-1,0)\}$, $k_1+k_2=7$, $k_1\geq k_2$.

\

\item \label{8u} If $L\subseteq \R$ is reduced with $|L|=8$ and $|L+L|\leq 22$, and $L$ is not contained in a single arithmetic progression, then either $L$ is output by WinnersSearch$(8,22)$ (runtime $1$ second), or is obtained by adding a single element to a set from item (\ref{7u}). Through brute force search in the latter case, we find $L=\{m+n\alpha: (m,n)\in P\}$ where $\alpha\notin \Q$ and $P$ takes one of the following forms up to similarity:  
\begin{enumerate}[(i)] \item $\{(0,0),\dots,(k_1-1,0),(0,1),\dots,(k_2-1,0)\}$, $k_1+k_2=8$, $k_1\geq k_2$,
\item $\{(0,0),\dots,(5,0),(7,0),(0,1)\}$ or $\{(0,0),\dots,(5,0),(0,1),(2,1)\}$,
\item $\{(0,0),\dots,(k_1-1,0),(0,1),\dots,(k_2-1,1), (0,2),\dots,(k_3-1,2)\}, \\ (k_1,k_2,k_3)\in \{(4,3,1),(3,3,2)\}.$
\end{enumerate}
 
\

\item \label{9u} If $L\subseteq \R$ is reduced with $|L|=9$ and $|L+L|\leq 26$, and $L$ does not contain $8$ elements in a single arithmetic progression, then either $L$ is output by WinnersSearch$(9,26)$ (runtime $7$ seconds), or can be obtained by adding a single element to a set from item (\ref{8u}). Through brute force search in the latter case, we find that $L=\{m+n\alpha: (m,n)\in P\}$ where $\alpha\notin \Q$ and $P$, up to similarity, takes one of the forms listed in item (3) of Lemma \ref{925sum}, or one of $14$ other forms, encoded as follows: $(7,1\sq\sq 1), (6\sq 1, 2),$ $(6,2\sq 1), (6\sq1, 1\sq 1), (6, 1\sq1\sq1), (5\sq1,3),(5,3\sq1), (4\sq1,4), (4,4,1),(4,4,\sq1),(3,4,2), (3,4,\sq2),\\(4,3,\sq2),(3,3,\sq3)$. Here the  three coordinate positions represent the rows $n=0,1,2$, respectively, the numbers in each coordinate represent lengths of arithmetic progressions with step $(1,0)$, and squares represent gaps. For example, $(6\sq1, 1\sq 1)$ encodes $P=\{(0,0),\dots,(5,0),(7,0),(0,1),(2,1)\}$.

\

\item \label{10u} If $L\subseteq \R$ is reduced with $|L|=10$ and $|L+L|\leq 30$, and $L$ does not contain $8$ elements in a single arithmetic progression, then either $L$ is output by WinnersSearch$(10,30)$ (runtime $1$ minute), or can be obtained by adding a single element to a set from item (\ref{9u}). Through brute force search in the latter case, we find that $L=\{m+n\alpha: (m,n)\in P\}$ where $\alpha\notin \Q$ and $P$, up to similarity, takes one of the forms listed in item (3) of Lemma \ref{1029sum}, which determine at most $29$ sums, or one of a few dozen other forms determining exactly $30$ sums, which we do not enumerate here but are easily computed and stored, and which are each similar to a configuration containing at least $9$ elements of $\{(m,n): 0\leq m \leq 7, n\in \{0,1\}\}$ or at least $9$ elements of $\{(m,n): 0\leq m \leq 3, 0\leq n \leq 2\}$.

\

\item \label{11u} If $L\subseteq \R$ is reduced with $|L|=11$ and $|L+L|\leq 34$, and $L$ does not contain $8$ elements in a single arithmetic progression, then either $L$ is output by WinnersSearch$(11,34)$ (runtime $6$ minutes), or can be obtained by adding a single element to a set from item (\ref{9u}). Through brute force search in the latter case, we find that $L=\{m+n\alpha: (m,n)\in P\}$ where $\alpha\notin \Q$ and $P$, up to similarity, takes one of the forms found (but not explicitly enumerated) in Section \ref{spsc}, which determine at most $33$ sums, or one of a few dozen other forms determining exactly $34$ sums, which we do not enumerate here but are easily computed and stored, and which are each similar to a configuration containing at least $9$ elements of $\{(m,n): 0\leq m \leq 7, n\in \{0,1\}\}$ or at least $9$ elements of $\{(m,n): 0\leq m \leq 3, 0\leq n \leq 2\}$.  

\end{enumerate}

As in the conclusion of Section \ref{ccs},  we collect all exceptional $(r,s)$ pairs from Lemmas \ref{tworows} and \ref{threerows}, and compute the number of distinct sums determined by $\{r^ms^n: (m,n)\in P\}$, this time for all $P$ collected in items (\ref{10u}) and (\ref{11u}) above. Once again, we find all $10$-element sets determine at least $31$ distinct sums, and all $11$-element sets determine at least 35 distinct sums, except scalings of $\{1, 2, 3, 4, 6, 8, 9, 12, 16, 18\}$ and  $\{1, 2, 3, 4, 6, 8, 9, 12, 16, 18, 24\}$, respectively, establishing the uniqueness claimed in Theorem \ref{mainNew}.

\section{Future work}

The path beyond $k=11$ appears treacherous. The best observed example for $k=12$ is $$A=\{1,2,3,4,6,8,9,12,16,18,24,32\},$$ which has $|AA|=35$, $|A+A|=41$, leading to the conjecture $SP(12)=41$. However, the gap between $40=12(3)+4$ and the territory covered by Theorems \ref{Fclass1} and \ref{Fclass2} is likely too wide for the methods of this paper to be effective on their own. One optimistic observation is that $A\subseteq (0,\infty)$ with $|A|=12$ and $|AA|\leq 41$ is indeed forced to be contained in a one or two-dimensional geometric progression. We establish this with the following more general result, which is morally similar to a lemma of Freiman \cite[p. 24]{Frei73}.

\begin{theorem}\label{Fdim} Suppose $L\subseteq \R$ with $|L|=k$ and $d\in \N$. If $|L+L|< (d+2)k-(d+1)(d+2)/2$, then, up to translation, either \begin{enumerate}[(i)] \item $L\subseteq \{a_1\alpha_1+\cdots +a_{d'}\alpha_{d'}: a_1,\dots,a_{d'}\in \Z\}$, $d'< d$, $\alpha_1,\dots,\alpha_d'\in \R$, or 
\item $L\subseteq \{a_1\alpha_1+\cdots +a_{d}\alpha_{d}: a_1,\dots,a_{d}\in \Z\}$, $0\in L$, and $\alpha_1,\dots,\alpha_{d}\in L$ linearly independent over $\Q$.
\end{enumerate}
\end{theorem}

\begin{proof} We induct on $d+k$. The result holds vacuously when $k=0$, as $0>-(d+1)(d+2)/2$, and  when $d=0$, as $|L+L|\geq 2k-1$ when $|L|=k$. These base cases suffice, but for nonvacuous base cases, the result is trivial when $k=1$ (take $\alpha_1$ to be the single element of $L$), and holds when $d=1$ by Theorem \ref{Fclass1}.  

\noindent Now suppose $d,k\geq 1$, and assume the result holds for the pairs $(d,k-1)$ and $(d-1,k-1)$. Suppose $L\subseteq \R$ with $|L|=k$ and $|L+L|< (d+2)k-(d+1)(d+2)/2$. For the remainder of the proof, for $\alpha_1,\dots,\alpha_i\in \R$, we let $P(\alpha_1,\dots,\alpha_i)=\{a_1\alpha_1+\cdots+a_{i}\alpha_{i}: a_1,\dots,a_{i}\in \Z\}$. 

\noindent After translation, assume $\min(L)=0=\alpha_0$, and iteratively define $\alpha_i$ as the minimum element of $L\setminus P(\alpha_1,\dots,\alpha_{i-1})$, terminating either with $i=d'<d$ and $L\subseteq P(\alpha_1,\dots,\alpha_{d'})$, in which case $L$ fits (i) and we are done, or with $i=d$.

\noindent In the latter case, first suppose $L\subseteq P(\alpha_1,\dots,\alpha_d)$. If $\alpha_1,\dots,\alpha_d$ are linearly independent over $\Q$, then $L$ fits (ii). If instead $\alpha_{j+1}=(a_1/b_1)\alpha_1+\cdots+(a_{j}/b_{j})\alpha_{j}$ for some $1\leq j \leq d-1$ with $a_1,\dots,a_{j}\in \Z$ and $b_1,\dots,b_j\in \N$, then $L\subseteq P(\alpha_1/b_1,\dots,\alpha_j/b_j,\alpha_{j+2},\dots,\alpha_d)$, so $L$ fits (i).

\noindent For what follows, suppose $L\not\subseteq P(\alpha_1,\dots,\alpha_d)$, let $\alpha_{d+1}$ be the minimum element of $L\setminus P(\alpha_1,\dots,\alpha_d)$, and let $L' = L\setminus\{0\}$, so $|L'|=k-1$. By construction, $0,\alpha_1,\dots,\alpha_{d+1} \in (L+L)\setminus (L'+L')$, hence  $$|L'+L'|\leq |L+L|-(d+2) < (d+2)(k-1)-(d+1)(d+2)/2.$$
\noindent Applying inductive hypotheses  for  $(d-1,k-1)$ and $(d,k-1)$, translating $L'$ if necessary, we have that one of the following holds:
\begin{enumerate}[(I)] \item $L'\subseteq P(\beta_1,\dots,\beta_{d'}), \ d'<d-1, \  \beta_1,\dots,\beta_{d'}\in \R, $
\item  $L'\subseteq P(\beta_1,\dots,\beta_{d-1}), \ \beta_1,\dots,\beta_{d-1}\in L'$ linearly independent over $\Q$
\item $L'\subseteq P(\beta_1,\dots,\beta_{d-1}), \  \beta_1,\dots,\beta_{d-1}\in \R, \ |L'+L'|\geq (d+1)(k-1)-d(d+1)/2,$
\item $L'\subseteq P(\beta_1,\dots,\beta_{d}), \ \beta_1,\dots,\beta_{d}\in L'$ linearly independent over $\Q$, $|L'+L'|\geq (d+1)(k-1)-d(d+1)/2$.
\end{enumerate} 
\noindent Let $\gamma$ be such that $L'\cup\{\gamma\}$ is a translation of $L$, which we now identify with $L$, so $\gamma=\min(L)$. In case (I), $L\subseteq P(\beta_1,\dots,\beta_{d'},\gamma)$, and $L$ fits (i). In case (II), $L\subseteq P(\beta_1,\dots,\beta_{d-1},\gamma)$, so if $\gamma$ is not in the $\Q$-span of $\beta_1,\dots,\beta_{d-1}$, then $L$ fits (ii). Otherwise, $L\subseteq P(\beta_1/b_1,\dots,\beta_{d-1}/b_{d-1})$ for $b_1,\dots,b_{d-1}\in \N$, so $L$ fits (i). 

\noindent Finally, we treat cases (III) and (IV) together. If $\gamma\in P(\beta_1,\dots,\beta_{d-1})$ (resp. $\gamma\in P(\beta_1,\dots,\beta_{d})$), then $L$ fits (i) (resp. (ii)). Otherwise, since $P$ is closed under addition and subtraction, $\gamma + x \notin L'+L'$ for all $x\in L'$, and also $2\gamma$ is less than all elements of $L'+L'$. In particular, $$|L+L|\geq |L'+L'|+k \geq (d+1)(k-1)-d(d+1)/2+ k = (d+2)(k) -(d+1)(d+2)/2, $$ contradicting our assumption on $|L+L|$ and completing the proof.
\end{proof}

\

\noindent Applying Theorem \ref{Fdim} on the product side with $d=2$ yields the following, which is applicable for $|A|=12$, $|AA|\leq 41$.

\begin{corollary} Suppose $A\subseteq (0,\infty)$ with $|A|=k$. If $|AA|\leq 4k-7$, then either $A$ is contained in a geometric progression or, up to scaling, $A\subseteq \{r^as^b: a,b\in \Z\},$ where $1,r,s\in A$ and $\log_r(s)\notin \Q$. 
\end{corollary}

\noindent Alas, $k=13$ appears to be the breaking point at which two-dimensional product structures no longer reign supreme for minimizing $\max\{|A+A|,|AA|\}$, as the best observed example is $$A=\{1,2,3,4,5,6,8,10,12,16,20,24,32\},$$ which satisfies $|A+A|=43$, $|AA|=46$, leading to the conjecture $SP(13)=46$. Note that this example can be written as a three-dimensional structure $\{2^a3^b5^c: (a,b,c)\in P\}$, where $$P=\{(0,0,0),(1,0,0),\dots,(5,0,0),(0,1,0),(1,1,0),(2,1,0),(3,1,0),(0,0,1),(1,0,1),(2,0,1) \},$$ and it is the union of three geometric progressions with the same common ratio $$\{1,2,4,8,16,32\}\cup\{3,6,12,24\}\cup\{5,10,20\},$$ but the starting points $\{1,3,5\}$ of the geometric progressions form an arithmetic, rather than geometric, progression. The development of new techniques, both computational and theoretical, to further establish precise values of $SP(k)$ remains appealing. 

\newpage

\noindent \textbf{Acknowledgements:} This research was initiated during the Summer 2025 Kinnaird Institute Research Experience at Millsaps College. All authors were supported during the summer by the Kinnaird Endowment, gifted to the Millsaps College Department of Mathematics. At the time of submission, all authors except Alex Rice  were Millsaps College undergraduate students. Alex Rice was partially supported by an AMS-Simons Research Enhancement Grant for PUI faculty. Figures \ref{fig1}-\ref{colfigure2} were created in tikz, aided by output from ChatGPT 4 and 5, provided by OpenAI; no other LLM assistance was utilized in the project. Python computations were completed using Google Colab and local Jupyter notebooks on personal laptops.

\end{document}